\documentclass[12pt,reqno]{article}

\usepackage[usenames]{color}
\usepackage{amssymb}
\usepackage{graphicx}
\usepackage{amscd}

\usepackage[colorlinks=true,
linkcolor=webgreen,
filecolor=webbrown,
citecolor=webgreen]{hyperref}

\definecolor{webgreen}{rgb}{0,.5,0}
\definecolor{webbrown}{rgb}{.6,0,0}

\usepackage{color}
\usepackage{fullpage}
\usepackage{float}

\usepackage{graphics,amsmath,amssymb}
\usepackage{amsthm}
\usepackage{amsfonts}
\usepackage{latexsym}
\usepackage{epsf}

\begin{document}

\theoremstyle{plain}
\newtheorem{theorem}{Theorem}
\newtheorem{corollary}[theorem]{Corollary}
\newtheorem{lemma}[theorem]{Lemma}
\newtheorem{proposition}[theorem]{Proposition}

\theoremstyle{definition}
\newtheorem{definition}[theorem]{Definition}
\newtheorem{example}[theorem]{Example}
\newtheorem{conjecture}[theorem]{Conjecture}
\newtheorem{problem}[theorem]{Problem}

\theoremstyle{remark}
\newtheorem{remark}[theorem]{Remark}

\begin{center}
\vskip 1cm{\LARGE\bf
Another Identity for Complete Bell Polynomials based on Ramanujan's Congruences
}
\vskip 1cm
\large
Ho-Hon Leung\\
Department of Mathematical Sciences \\
United Arab Emirates University \\
Al Ain, 15551\\
United Arab Emirates\\
\href{mailto:hohon.leung@uaeu.ac.ae}{\tt hohon.leung@uaeu.ac.ae}\\
\end{center}

\begin{abstract}
Let $p(n)$ be the number of partitions of a positive integer $n$. We derive a new identity for complete Bell polynomials based on a generating function of $p(7n+5)$ given by Ramanujan.
\end{abstract}

\section{Introduction}  \label{section1}

Let $(x_1, x_2, \dots)$ be a sequence of real numbers. The {\it partial} exponential Bell polynomials are polynomials given by 
\begin{align*}
B_{n,k}(x_1,x_2,\dots,x_{n-k+1}) &= \sum_{\pi(n,k)} \frac{n!}{j_1 ! j_2 ! \dots j_{n-k+1}!} \Big(\frac{x_1}{1!} \Big)^{j_1} \Big(\frac{x_2}{2!} \Big)^{j_2}\dots \Big(\frac{x_{n-k+1}}{(n-k+1)!} \Big)^{j_{n-k+1}}
\end{align*}where $\pi (n,k)$ is the positive integer sequence $(j_1, j_2, j_{n-k+1})$ satisfies the following equations: 
\begin{align*}
j_1 +j_2 +\dots +j_{n-k+1} &=k, \\
j_1 + 2 j_2 + \dots + (n-k+1) j_{n-k+1} &=n.
\end{align*}For $n\geq 1$, the $n^{\text{th}}$-complete exponential Bell polynomial $B_n(x_1, \dots, x_n)$ is the following:
\begin{align*}
  B_n (x_1\dots x_n)  &=\sum_{k=1}^n B_{n,k} (x_1, \dots , x_{n-k+1}).
\end{align*}The complete exponential Bell polynomials can also be defined by power series expansion as follows:
\begin{align}
  \label{equation1}  \text{exp}\Big( \sum_{m=1}^\infty x_m \frac{t^m}{m!} \Big) &= \sum_{n=0}^\infty B_n (x_1, \dots , x_n) \frac{t^n}{n!},
\end{align}where $B_0\equiv 1$. Bell polynomials were first introduced by Bell \cite{Bell}. The books written by Comtet \cite{Comtet} and Riordan \cite{Riordan} serve as excellent references for the numerous applications of Bell polynomials in combinatorics. 

Let $(a;q)_n$ be the $q$-Pochhammer symbol for $n\geq 1$. That is, 
\[ (a;q)_n:=\prod_{k=0}^{n-1} (1-aq^k)=(1-a)(1-aq)(1-aq^2)\dots (1-aq^{n-1}).\]Considered as a formal power series in $q$, the definition of $q$-Pochhammer symbol can be extended to an infinite product. That is, \[(a;q)_\infty := \prod_{k=0}^\infty (1-aq^k).\]We note that $(q;q)_\infty$ is the Euler's function. Let $p(n)$ be the number of partitions of $n$. The generating function of $p(n)$ can be written as 
\begin{align*}
 \sum_{n=0}^\infty p(n) q^n &= \frac{1}{(q;q)_\infty}.
\end{align*}Andrew's book \cite{Andrews} serves as an excellent reference to the theory of partitions. 

Ramanujan's congruences are congruence properties for $p(n)$: 
\[ p(5k+4) \equiv 0  \text{ (mod $5$)};  \quad p(7k+5)\equiv 0 \text{ (mod $7$)}; \quad p(11k+6)\equiv 0 \text{ (mod $11$)}.\]In 1919, Ramanujan \cite{Ramanujan} proved the first two congruences by the following two identities: 
\begin{align}
\label{equation2}\sum_{k=0}^\infty p(5k+4) q^k &= 5 \frac{(q^5;q^5)_{\infty}^5}{(q;q)_{\infty}^6}, \\
\label{equation3}\sum_{k=0}^\infty p(7k+5) x^k &= 7\frac{(q^7;q^7)_{\infty}^3}{(q;q)_{\infty}^4}+ 49x \frac{(q^7;q^7)_{\infty}^7}{(q;q)_{\infty}^8}.
\end{align}

Bouroubi and Benyahia-Tani \cite{Sadek} proved an identity for complete Bell polynomials based on (\ref{equation2}). As an analogue to their result, we give an identity for complete Bell polynomials based on (\ref{equation3}). In other words, we derive formulas that relate $p(7n+5)$ and certain complete Bell polynomials.

\section{Main Theorem} \label{section2}

Let $\sigma(n)$ be the sum of divisors (including $1$ and $n$) for $n$. It is well known that $\sigma(n)$ is a multiplicative function. That is, if $n$ and $m$ are coprime, then
\begin{align}
\label{equation4} \sigma(mn)&=\sigma(m) \sigma(n).
\end{align}If $m\geq 1$, then 
\begin{align}
\label{equation5} \sigma(p^m)&=1+p+\dots +p^{m-1}+p^m=\frac{p^{m+1}-1}{p-1}.
\end{align}

\begin{lemma} \label{lemma1}
Let $n=7^m n'$ such that $m\geq 1$ and $\text{g.c.d.}(n,n')=1$. Then 
\begin{align*}
\sigma(n) &= \frac{7^{m+1}-1}{7^m -1} \sigma\big( \frac{n}{7}\big).
\end{align*}
\end{lemma}

\begin{proof}By (\ref{equation4}) and (\ref{equation5}),
\begin{align}
\label{equation6} \sigma(n) &= \sigma (7^m) \sigma(n')=\frac{7^{m+1}-1}{6} \sigma\Big(\frac{n}{7^m}\Big), \\
\label{equation7} \sigma\Big(\frac{n}{7} \Big) &= \sigma \Big( 7^{m-1} \Big) \sigma\Big(\frac{n}{7^m} \Big) =\frac{7^m -1}{6}\sigma\Big( \frac{n}{7^m}\Big).
\end{align}By a combination of (\ref{equation6}) and (\ref{equation7}), we get the desired result.
\end{proof}

\begin{theorem}
Let $n\geq 1$. We write $n=7^m n'$ where $m\geq 0$ and $\text{g.c.d.}(n,n')=1$. Let $d_n$ and $e_n$ be the following sequences of numbers respectively:
\begin{align*}
d_n &= \frac{\sigma(n)}{n} \Big(1+\frac{18}{7^{m+1}-1}\Big),\\
e_n &= \frac{\sigma(n)}{n}\Big( 1+\frac{42}{7^{m+1}-1} \Big).
\end{align*} Then we have the following identity: 
\begin{align*}
 7B_n(1!d_1, 2!d_2,\dots, n!d_n)+49n B_{n-1}(1!e_1, 2!e_2 , \dots, (n-1)!e_{n-1}) &= n! p(7n+5).
\end{align*}
\end{theorem}

\begin{proof}
Let $G(x)$ and $H(x)$ be the following functions:
\begin{align}
\label{equation8} G(x) &= 7 \frac{(x^7 ; x^7)_\infty^3}{(x;x)_\infty^4}, \\
\label{equation9} H(x) &= 49x \frac{(x^7;x^7)_\infty^7}{(x;x)_\infty^8}.
\end{align}$G(x)$ and $H(x)$ are well-defined on the interior of the unit disk in the complex plane by analytic continuation. We get the following two equations by (\ref{equation8}) and (\ref{equation9}), 
\begin{align*}
\ln (G(x)) &= \ln 7 -3 \sum_{i=1}^\infty \ln(1-x^{7i}) +4 \sum_{i=1}^\infty \ln (1-x^i), \\
 \ln (H(x)) &= \ln 49 +\ln x -7 \sum_{i=1}^\infty \ln (1-x^{7i}) + 8 \sum_{i=1}^\infty \ln (1-x^i).
\end{align*}By using the power series expansion of $\ln (1-x)$, we get 
\begin{align}
\label{equation10} \ln (G(x)) &= \ln 7-3 \sum_{i=1}^\infty \sum_{j=1}^\infty \frac{x^{7ij}}{j} +4\sum_{i=1}^\infty \sum_{j=1}^\infty \frac{x^{ij}}{j}, \\
\label{equation11} \ln (H(x)) &= \ln 49 +\ln x -7\sum_{i=1}^\infty \sum_{j=1}^\infty \frac{x^{7ij}}{j} +8 \sum_{i=1}^\infty \sum_{j=1}^\infty \frac{x^{ij}}{j}.
\end{align}Let $f_1(x)$ and $f_2(x)$ be the following two functions: 
\begin{align}
\label{equation12}  f_1(x) &= \sum_{i=1}^\infty \sum_{j=1}^\infty \frac{x^{7ij}}{j}, \\
\label{equation13} f_2(x) &= \sum_{i=1}^\infty \sum_{j=1}^\infty \frac{x^{ij}}{j}.
\end{align}$f_1(x)$ has non-zero coefficients for $x^m$ if and only if $m$ is a multiple of $7$. More precisely, let 
\begin{align*}
f_1(x) &= \sum_{i=1}^\infty a_i x^i.
\end{align*}Then, 
\begin{align}
 \label{equation14} a_i &=
\begin{cases}
    \frac{\sigma(i/7)}{i/7},& \text{if }  7| i; \\
    0,              & \text{otherwise.}
\end{cases}
\end{align}We write $f_2(x)$ as 
\begin{align*}
  f_2(x) &= \sum_{i=1}^\infty b_i x^i
\end{align*}where 
\begin{align}
 \label{equation15} b_i &= \frac{\sigma(i)}{i}.
\end{align}By (\ref{equation10}), (\ref{equation12}), (\ref{equation13}), (\ref{equation14}), (\ref{equation15}), we get 
\begin{align}
\label{equation16} \ln (G(x))  &=\ln 7 +\sum_{i=1}^\infty d_i x^i 
\end{align}where 
\begin{align}
\label{equation17}  d_i &= 
\begin{cases}
    \frac{4\sigma(i)}{i}-\frac{3\sigma(i/7)}{i/7},& \text{if } 7|i;\\
    \frac{4\sigma(i)}{i},              & \text{otherwise.}
\end{cases}
\end{align}By Lemma \ref{lemma1}, we write (\ref{equation17}) as 
\begin{align}
\label{equation18} d_i &= \frac{\sigma(i)}{i} \Big( 1+\frac{18}{7^{m+1}-1} \Big)
\end{align}for $i=7^m i'$, $m\geq 0$. Similarly, by (\ref{equation11}), (\ref{equation12}), (\ref{equation13}), (\ref{equation14}), (\ref{equation15}), we get
\begin{align}
\label{equation19} \ln(H(x)) &= \ln 49 +\ln x + \sum_{i=1}^\infty e_i x^i
\end{align}where
\begin{align}
\label{equation20} e_i &= 
\begin{cases}
    \frac{8\sigma(i)}{i}-\frac{7\sigma(i/7)}{i/7},& \text{if } 7|i;\\
    \frac{8\sigma(i)}{i},              & \text{otherwise.}
\end{cases}
\end{align}By Lemma \ref{lemma1}, we write (\ref{equation20}) as 
\begin{align}
\label{equation21} e_i &= \frac{\sigma(i)}{i} \Big( 1+\frac{42}{7^{m+1}-1} \Big)
\end{align}for $i=7^m i'$, $m\geq 0$. By (\ref{equation1}), (\ref{equation16}), (\ref{equation18}), we have
\begin{align}
\nonumber G(x)&=\exp(\ln G(x)) =7\exp \Big(\sum_{n=1}^\infty d_n x^n  \Big)=7\exp\Big(\sum_{n=1}^\infty (n! d_n)\frac{x^n}{n!} \Big)  \\
\label{equation22}  &= 7\Big(\sum_{n=0}^\infty B_n (1! d_1, 2! d_2, \dots, n! d_n)\frac{x^n}{n!}\Big).
\end{align}By (\ref{equation1}), (\ref{equation19}), (\ref{equation21}), we have
\begin{align}
\nonumber H(x) &= \exp(\ln H(x)) = 49x \exp \Big(\sum_{n=1}^\infty e_n x^n \Big)=49x\exp \Big( \sum_{n=1}^\infty (n! e_n)\frac{x^n}{n!} \Big) \\
\nonumber  &= 49x \Big( \sum_{n=0}^\infty B_n (1!e_1, 2!e_2, \dots, n!e_n)\frac{x^n}{n!} \Big) \\
\label{equation23}  &= 49 \Big(\sum_{n=1}^\infty B_{n-1}(1!e_1, 2!e_2,\dots, (n-1)! e_{n-1})\frac{x^n}{(n-1)!}  \Big).
\end{align}By (\ref{equation3}), (\ref{equation8}), (\ref{equation9}), (\ref{equation22}), (\ref{equation23}), we have the following identity:
\begin{align}
\nonumber \sum_{n=0}^\infty p(7n+5) x^n&= 7+\sum_{n=1}^\infty \Big( 7B_n(1!d_1, 2!d_2,\dots, n!d_n)+49n B_{n-1} (1!e_1, 2!e_2, \dots, (n-1)! e_{n-1})     \Big) \frac{x^n}{n!} 
\end{align}as desired.
\end{proof}

\section{Acknowledgement}
The author is grateful to Victor Bovdi for pointing out some references about this topic. The author is also grateful to Mohamed El Bachraoui for some valuable discussions on the topic. The author is supported by Startup-Grant-2016 (fund no. 31S263) from United Arab Emirates University.

\bigskip
\hrule
\bigskip

\noindent 2010 {\it Mathematics Subject Classification}: 05A17, 11P81.

\noindent \emph{Keywords: } Bell polynomials, integer partitions, divisor sums, Ramanujan's congruences 

\bigskip
\hrule
\bigskip

\end{document}